\documentclass[10pt]{amsart}
\usepackage{amssymb,latexsym, mathrsfs}
\usepackage{xypic}
\usepackage[all]{xy}
\usepackage{amscd}

\usepackage{color}

 \RequirePackage{amsbsy} 
 \RequirePackage{amsopn}  
 \RequirePackage{amsmath}
  \RequirePackage{amssymb}

  \usepackage[mathscr]{eucal}
\definecolor{darkmagenta}{rgb}{0.5, 0, 0.5}
\definecolor{darkblue}{rgb}{0.1, 0.05, 0.9}
\definecolor{darkgreen}{rgb}{0.05, 0.35, 0.6}
\definecolor{darkred}{rgb}{0.9, 0.1, 0.3}

	  \newcommand{\wS}{\wedge_{\{S\}}}
	  \newcommand{\FF}{\overline{\boldsymbol{F}}}
	  \newcommand{\F}{ {\boldsymbol{F}}}
	  \newcommand{\f}{ {\boldsymbol{f}}}

	  \newcommand{\ext}{\mbox{\rm\texttt{ext}}}
	  \newcommand{\e}{\mbox{\rm\tiny\texttt{e}}}
	
	  \newcommand{\dd}{\mbox{\rm\texttt{d}}}
	    \newcommand{\con}{\mbox{\rm\texttt{con}}}
	  \newcommand{\co}{\mbox{\rm\tiny\texttt{c}}}
	
	   \newcommand{\id}{\mbox{\rm\texttt{id}}}
	    
	       \newcommand{\vv}{\mbox{\rm\texttt{v}}}
	        \newcommand{\ww}{\mbox{\rm\texttt{w}}}
	            \newcommand{\ttt}{\mbox{\rm\texttt{t}}}
	            \newcommand{\ttts}{\mbox{\rm\tiny\texttt{t}}}
	         \newcommand{\wws}{\mbox{\rm\tiny\texttt{w}}}
	            \newcommand{\vvs}{\mbox{\tiny\rm\texttt{v}}}

	     \newcommand{\stf} {\star{_{\!{_f}}}}
\newcommand{\stt} {\widetilde{\star}}
\newcommand{\stu} {\overline{\star}}
\newcommand{\sta} {\star_{\mbox{\it\tiny\texttt{a}}}}

\newcommand{\SStar}{{\mbox{\rm\texttt{SStar}}}}
\newcommand{\SStarf}{\mbox{\rm\texttt{SStar}}_{\!\mbox{\tiny\it\texttt{f}}}}

\newcommand{\SStarstab}{\overline{\mbox{\rm\texttt{SStar}}}}

\newcommand{\SStarstabft}{\widetilde{\mbox{\rm\texttt{SStar}}}}
\newcommand{\SStarsp}{\mbox{\rm\texttt{SStar}}_{\!\mbox{\tiny\it\texttt{sp}}}}

\newcommand{\SStareabf}{\mbox{\rm\texttt{SStar}}_{\!\mbox{\tiny\it\texttt{f,eab}}}}
\newcommand{\SStareab}{\mbox{\rm\texttt{SStar}}_{\!\mbox{\tiny\it\texttt{eab}}}}
\newcommand{\SStarval}{\mbox{\rm\texttt{SStar}}_{\!\mbox{\tiny\it\texttt{val}}}}

\theoremstyle{plain}

\newtheorem{theorem}{Theorem}[section]

\newtheorem{corollary}[theorem]{Corollary}

\newtheorem{lemma}[theorem]{Lemma}

\newtheorem{proposition}[theorem]{Proposition}
\newtheorem{remark}[theorem]{Remark}

\DeclareMathOperator{\Spec}{Spec}
\DeclareMathOperator{\QSpec}{QSpec}

\DeclareMathOperator{\Max}{Max}
\DeclareMathOperator{\QMax}{QMax}

\DeclareMathOperator{\Assp}{Assp}

\newcommand{\sub}{\subseteq}

\begin{document}

\title[Intersectionns of quotient rings  and  P$\vv$MDs]{Intersections of quotient rings  and\\  Pr\"ufer $\vv$-multiplication domains}
\author[S. El Baghdadi]{Said El Baghdadi}
\address{(S. El Baghdadi) Department of Mathematics,
Beni Mellal University, 23000 Beni Mellal, Morocco}
\email{baghdadi@fstbm.ac.ma}

\author[M. Fontana]{Marco Fontana}
\address{(M. Fontana) Dipartimento di Matematica e Fisica,
Universit\`a degli Studi ``Roma Tre'', 00146 Rome, Italy}
\email{fontana@mat.uniroma3.it}

\author[M. Zafrullah]{Muhammad Zafrullah}
\address{ (M. Zafrullah) Department of Mathematics and Statistics, Idaho State University, Pocatello, ID 83209-8085, USA}
\email{mzafrullah@usa.net}

\thanks{$2010$ Mathematics Subject Classification: 13F05, 13G05, 13B30, 13E99}
 
\thanks{Words and phrases:  defining family of overrings; finite character;  Mori, strong Mori and Krull domain; star and semistar operation; $\vv$-, $\ttt$-, $\ww$-operation; Pr\"ufer and Dedekind domain; Pr\"ufer $\vv$-multiplication domain. }
\thanks{The second named author thanks the ``National Group for Algebraic and Geometric Structures, and their Applications'' (GNSAGA) of the {\sl Istituto Nazionale di Alta Matematica} for a partial support.}

\begin{abstract}
Let $D$ be an integral domain with quotient field $K$. Call an overring $S$ of
$D$ a subring of $K$  containing $D$ as a subring.   A family $\{S_\lambda\mid\lambda \in \Lambda \}$ of overrings of $D$ is called
a defining family of $D$, if $D = \bigcap\{S_\lambda\mid\lambda \in \Lambda \}$. 
Call an overring $S$ a sublocalization of $D$, if $S$ has a defining family consisting of rings of fractions of $D$. 
Sublocalizations and their intersections exhibit interesting examples of semistar or star operations \cite{And}. We show as
a consequence of our work that domains that are locally finite intersections of
Pr\"ufer $\vv$-multiplication  (respectively, Mori) sublocalizations turn out to be Pr\"ufer $\vv$-multiplication domains (respectively, Mori);  in particular, for the Mori domain case, we reobtain  a special case of \cite[Th\'eor\`eme 1]{Q} and \cite[Proposition 3.2]{de}.
We also show
that, more than the finite character of the defining family, it is the finite character
of the star operation induced by the defining   family 
that causes  the interesting results. As a particular case of this theory,  we provide a purely algebraic approach for characterizing Pr\"ufer $\vv$-multiplication domains as a subclass of the class of essential domains (see also \cite[Theorem 2.4]{FT}).
\end{abstract}

\maketitle

\section{Introduction}

Throughout this note $D$ denotes an integral domain and $K$ its quotient field.

A family of overrings (rings between $D$ and $K$) $\{S_\lambda \mid \lambda \in \Lambda \}$ of $D$ such that  $D = \bigcap\{S_\lambda \mid \lambda \in \Lambda \}$ is called a {\it defining family of $D$.}   We say that a {\it defining family} is {\it locally finite} (or, {\it has finite character}) if every nonzero element of $D$ is a unit   in all but a finite number of the $S_\lambda$'s.

 When the
rings $S_\lambda$ are quotient rings of $D$, we get a representation of $D$ as an intersection of
quotient rings. This is the case of an important class of classical domains, e.g., the class of essential domains (definition recalled later), which includes Dedekind domains, Krull domains, Pr\"ufer domains and their generalization Pr\"ufer $\vv$-multiplication domains ({for short P$\vv$MD,  definition recalled later).
A more general inte\-re\-sting representation is when each $ S_\lambda  \in \{S_\lambda \mid \lambda \in \Lambda \}$ is itself an intersection of quotient rings, e.g., if each $ S_\lambda $ is ($\ttt$)-flat over $D$ (definition recalled later).

In this note, we study some of these representations  defined by appropriate finite character type conditions.

The theory of star   and semistar  operations   is one of the key ingredients in achie\-ving  this goal. 
In fact,  any representation associated to a defining family of a domain $D$
induces a star operation on $D$ (see \cite{And}).  More generally,   any intersection of overrings of $D$ defines  a semistar operation on $D$ (see for instance \cite{FH}).   In the present paper,   we will mainly use the last more general setting.

The aim of this paper is to shed new light on some questions in the literature  related to representations of
domains as intersections of quotient rings. 
 For instance, a well known result in this area is  a simple and elegant characterization, in the $\ttt$-finite character  case, of the P$\vv$MDs
 given by M. Griffin   in \cite{Gri 1}: 
  they are exactly the
essential domains.   An example by  W. Heinzer and J. Ohm  \cite{H-O} shows that there exist essential domains that are not P$\vv$MDs.
 The question of when an essential
domain is a P$\vv$MD was solved recently in \cite{FT}, by
using topological methods.
In this paper, we introduce a weak form of the finite character property of a defining family of a domain which turns out to be the key idea for an algebra theoretic proof of the question of when an essential domain is a P$\vv$MD.

 In Sections 2 and 3, we give an overview on the theory of semistar operations
and its interaction with a representation of a domain as an intersection  of overrings.   In Section 4,
we  investigate the question of when an intersection of a family of P$\vv$MDs is a P$\vv$MD. 
In the case of
an  intersection of overrings with finite character,   we give an affirmative answer to the previous question, providing a generalization of a similar well
known fact concerning the Krull domains,   i.e., a locally finite intersection of Krull domains is a Krull domain. 
 In Section 5, we provide a purely algebraic approach for characterizing P$\vv$MDs as a subclass of the class of essential domains.

\bigskip

\section{Preliminaires}

Throughout this paper, let $D$ be an integral domain with quotient field $K$. Let $\FF(D)$ (respectively, $\F(D)$; $\f(D)$) be the set of all nonzero $D$--submodules of $K$ (respectively, nonzero fractional ideals; nonzero finitely
generated fractional ideals) of $D$ (thus, $\f(D)\subseteq\F(D)\subseteq\FF(D)$).

A mapping $\star:\FF(D)\longrightarrow\FF(D)$, $E\mapsto E^\star$, is called a \emph{semistar operation} of $D$ if, for all $z\in K$, $z\neq 0$ and for all $E,F \in\FF(D)$, the following properties hold: $\mathbf{\bf(\star_1)} \;(zE)^\star =zE^\star$; $\mathbf{\bf (\star_2)} \; E\subseteq F \Rightarrow E^\star \subseteq F^\star$; $\mathbf{ \bf (\star_3)}\; E \subseteq E^\star$; and $\mathbf{ \bf (\star_4)}\;  E^{\star \star} := (E^\star)^\star = E^\star $.

When $D^\star=D$,  $\star$  is called a
\emph{(semi)star operation} on $D$; in this case, the restriction of $\star$ to $\boldsymbol{F}(D)$ is a usual
\emph{star operation} (see \cite[Section 32]{Gil} for more details).

As in the classical star-operation setting, we associate to a
semistar operation $\star$ of $D$ a new semistar operation
$\stf$ of $D$ by setting, for every $E\in\FF(D)$,
\begin{equation*}
E^{\stf} := \bigcup \{F^\star \mid  F \subseteq E,  F \in \f(D)\}.
\end{equation*}
We call $\stf$ the semistar operation of finite type of $D$ \emph{associated }to $\star$. If $\star=\stf$, we say that $\star$ is a
 \emph{semistar operation of finite type}  on $D$. Note that $(\stf)_{_{\! f}} = \stf$, so $ \stf$ is a semistar operation
of finite type of $D$.

\smallskip

We denote by $\SStar(D)$ (respectively, $\SStarf(D)$) the set of all semistar ope\-ra\-tions (respectively, semistar operations of finite type) on $D$.
Given two semistar operations $\star'$ and $\star''$ of $D$, we say that $\star' \leq\ \star''$ if $E^{\star'} \subseteq E^{\star''}$, for all $E \in\FF(D)$. The relation ``$\leq$'' introduces a partial ordering in $\SStar(D)$. From the definition of $\stf$, we deduce that $\stf \leq\star$ and that $\stf$ is the largest semistar operation of finite type smaller than or equal to $\star$.

\smallskip

A semistar operation $\star$ defined on an integral domain $D$ is called \emph{stable} provided that, for any $E,H\in \FF(D)$, we have $(E\cap H)^\star=E^\star\cap H^\star$. We denote by $\SStarstab(D)$ the set of stable semistar operations on $D$.

Given a semistar operation $\star$ on $D$, we can always associate to $\star$ a stable semistar operation $\stu$ by defining, for every $E\in\FF(D)$,
\begin{equation*}
E^{\stu}:=\bigcup \{(E:I)\mid I \mbox{ nonzero ideal of } D \mbox{ such that } I^\star =D^\star\}.
\end{equation*}
It is easy to see that $\stu \leq \star$ and, moreover, that $\stu$ is the largest stable semistar operation that precedes $\star$. Therefore, $\star$ is stable if and only if $\star = \stu$ \ \cite[Proposition 3.7, Corollary 3.9]{FH}.

 As in the case of  $\stu$, we can associate to each semistar operation $\star$ a stable semistar operation of finite type $\stt$ by defining, for every $E\in\FF(D)$,
\begin{equation*}
\begin{array}{rl}
E^{\stt}:= & \bigcup \{ (E:J)\mid J \mbox{ nonzero finitely generated ideal of } D \\
& \hskip 70pt \mbox{ such that }J^\star =D^\star\}.
\end{array}
\end{equation*}
The stable semistar operation of finite type $\stt$ is smaller   than or equal to  $\star$, and it is the biggest stable semistar operation of finite type smaller  than or equal to $\star$. It follows that $\star$ is stable of finite type if and only if $\star=\stt$.
We denote by $\SStarstabft(D)$   the set of stable semistar operations of finite type on $D$.

\medskip

Let $\boldsymbol{\mathscr{S}}:= \{S_\lambda \mid \lambda \in \Lambda\}$ be a nonempty family of overrings of an integral domain $D$.
Let $\wedge_{\boldsymbol{\mathscr{S}}}$ be the semistar operation on $D$ defined, for each $E \in \FF(D)$, by:
$$
E^{\wedge_{\boldsymbol{\mathscr{S}}}} := \bigcap \{ES_\lambda \mid \lambda \in \Lambda\}\,.
$$

In particular, if  $S$ be an overring of $D$ and $\boldsymbol{\mathscr{S}} :=\{S\}$, then the operation  $\wS$  is a semistar operation of finite type.
If $S$ is   a $D$-flat overring, then $\wS$ is a semistar operation stable (and of finite type) and conversely (see \cite[Theorem 7.4(i)]{Ma} and \cite[Proposition 1.7]{U}).
In general, for each nonempty family $\boldsymbol{\mathscr{S}}$ of $D$-flat overrings of $D$, $\wedge_{\boldsymbol{\mathscr{S}}}$ is stable, but it is not necessarily of finite type.

\medskip

\smallskip

If $Y$ is a nonempty subset of the prime spectrum $\Spec(D)$ of an integral domain $D$, then we define the semistar operation $\mbox{\texttt{s}}_Y$ \emph{induced} by $Y$ as the semistar operation associated to the set $\boldsymbol{\mathcal{T}}(Y) := \{D_P \mid P \in Y\}$, i.e.,  $\mbox{\texttt{s}}_Y := \wedge _{\boldsymbol{\mathcal{T}}(Y) }$ is the semistar operation defined by
\begin{equation*}
E^{\mbox{\tiny\texttt{s}$_Y$}}:=\bigcap \{ED_{P} \mid P\in Y\},\;  \mbox{ for every } E\in \FF(D).
\end{equation*}
A semistar operation of the type $\mbox{\texttt{s}}_Y$, for some $Y \subseteq \Spec(D)$, is called a \emph{spectral} semistar operation on $D$.
 We denote by $\SStarsp(D)$ the set of spectral semistar operations on $D$. Clearly, a spectral semistar operation is stable, i.e., $\SStarsp(D) \subseteq $ $\SStarstab(D)$. Moreover, it is known that  the previous sets of semistar operations coincide in the finite type case (see for instance \cite[Lemma 1.32]{Pi}):
 $$\SStarsp(D) \cap \SStarf(D)  = \SStarstab(D) \cap \SStarf(D) =\SStarstabft(D)\,.
 $$

\smallskip

For star operations $\ast$,   the notion of a ``star-ideal''  (that is, a nonzero ideal
$I$ of $D$, such that $I^\ast= I$) is very useful.  For a semistar operation $\star$, we need a more general notion,
that coincides with the notion of star-ideal, when $\star$ is a (semi)star operation.
We say
that a nonzero  ideal $I$ of $D$ is a quasi-$\star$-ideal if $I^\star
\cap D = I$. For example,
it is easy to see that, for each nonzero ideal $I$ of $D$ such that $ I^\star \cap D \neq D$, then  $J:= I^\star \cap D$  is a quasi-$\star$-ideal of $D$ that contains $I$;
in particular, a $\star$-ideal (i.e., a nonzero ideal $I$ such that $I^\star = I$) is a quasi-$\star$-ideal. Note that $I^\star \cap D \neq D$
is equivalent to
$I^\star \neq D^\star$.    A  {\it quasi-$\star$-prime} is a  quasi-$\star$-ideal which is also a prime ideal.
We call a {\it quasi-$\star$-maximal} a maximal element in the set of all proper quasi-$\star$-ideals of $D$. We denote
by $\QSpec^\star(D)$ (respectively, $\QMax^\star(D)$) the set of all quasi-$\star$-primes (respectively,
quasi-$\star$-maximals) of $D$.
It is well known that a quasi-$\star$-maximal ideal is a prime ideal and it is possible
to prove that each quasi-$\stf$-ideal is contained in a quasi-$\stf$-maximal ideal (see for instance \cite[Lemma  2.3]{FL-2003}).
When $\star$ is a (semi)star operation, we simply set   $\Max^\star(D)$ (respectively,  $\Spec^\star(D)$)  instead of $\QMax^\star(D)$ (respectively, $\QSpec^\star(D)$).

\smallskip

A semistar operation $\star$ on an integral domain $D$ is said to be an \emph{\texttt{eab} semistar operation} (respectively, an \emph{\texttt{ab} semistar operation}) if, for every $F,G,H\in\f(D)$ (respectively, for every $F\in\f(D)$, $G,H\in\FF(D)$) the inclusion $(FG)^\star\subseteq(FH)^\star$ implies $G^\star\subseteq H^\star$.
Note that, if $\star$ is \texttt{eab}, then $\stf$ is also \texttt{eab}, since $\star$ and $\stf$ agree on nonzero finitely generated fractional ideals.
We can associate to any semistar ope\-ration $\star$ of
$D$  an \texttt{eab} semistar operation of finite type
$  \star_a $  of $ D $, called {\it the \texttt{eab}
semistar
operation associated to $ \star  $},   defined as follows
for each $ F \in \boldsymbol{f}(D)$ and
for each  $E \in {\overline{\boldsymbol{F}}}(D)$:
$$
\begin {array} {rl}
F^{\star_a} :=& \hskip -5pt  \bigcup\{((FH)^\star:H^\star) \; \ | \; \, \; H \in
\boldsymbol{f}(D)\}\,, \\
E^{\star_a} :=& \hskip -5pt  \bigcup\{F^{\star_a} \; | \; \, F \subseteq E\,,\; F \in
\boldsymbol{f}(D)\}\,,
\end{array}
$$
\cite[Definition 4.4 and Proposition 4.5]{FL}. \rm
 The previous
construction, in the ideal systems setting, is essentially due to {P. Jaffard}  \cite{J:1960}
and {F. Halter-Koch} \cite{HK:1998}.

 Obviously
 $(\star_{_{\!f}})_{a}= \star_{a}$. Note also that, when $\star =
\star_{_{\!f}}$, then $\star$ is \texttt{eab} if and only if $\star =
\star_{a}\,$ \cite[Proposition 4.5(5)]{FL}.

\medskip

 A \emph{valuative semistar operation} is a semistar operation  of the type $\wedge_{\boldsymbol{\mathcal W}}$, where $\boldsymbol{\mathcal W}$ is a family of valuation overrings of $D$; it is easy to see that $\wedge_{\boldsymbol{\mathcal W}}$ is an \texttt{eab} semistar operation. In particular, if $\boldsymbol{\mathcal V}$ is the set of all valuation overrings of $D$, the $\mbox{\texttt{b}}$-operation, where $\mbox{\texttt{b}} :=
\wedge_{\boldsymbol{\mathcal V}}$, is an \texttt{eab} semistar operation of finite type on $D$ (see  \cite[pages 394 and 398]{Gil} and \cite[Proposition 4.5]{FiSp}).

Just as in the case of the relation between stable and spectral operations, not every \texttt{eab} semistar operation is valutative, but the two definitions agree on finite type operations (see, for instance, \cite[Corollaries 3.8 and 5.2]{FL}).

Denote by $\SStarval(D)$ (respectively, $\SStareab(D)$;   $\SStareabf(D)$) the set of valutative (respectively, \texttt{eab};  \texttt{eab} of finite type) semistar operations on $D$. By the previous remarks, we have:
$$
\SStareabf(D) :=\SStareab(D) \cap\SStarf(D) = \SStarval(D) \cap\SStarf(D)\,.
$$

\bigskip

\section{Sublocalizations and associated semistar operations}

Let $D$ be an integral domain and $S$ an overring of $D$.
It is possible to define an ``extension'' map $\ext:= \ext(D,S): \SStar(D) \rightarrow \SStar(S)$ (respectively,  ``contraction'' map $\con:= \con(S,D): \SStar(S) \rightarrow \SStar(D)$, by setting $ \star \mapsto \star^{\e}$ (respectively, $ \star \mapsto \star^{\co}$), where:
$$
\begin{array}{rl}
{Ê} & \star^{\e} :    \FF(S)  	\subseteq  \FF(D) \overset{ \star}{\rightarrow} \FF(D) \overset{ \, \otimes_D S}{\longrightarrow} \FF(S),
  \;\; F \mapsto (F^\star)S\,,\\

(\mbox{respectively,}Ê&\star^{\co} : \FF(D) \overset{ \, \otimes_D S}{\longrightarrow} \FF(S) \overset{ \star}{\rightarrow} \FF(S)\subseteq \FF(D),\;\; E \mapsto (ES)^\star ).
\end{array}
$$
Note that $(F^\star)S = F^\star \in \FF(S)$,  since for each nonzero $s \in S$, $sF^\star = (sF)^\star  \subseteq F^\star$, being $F \in \FF(S)$.

We collect in the following lemma some basic properties of the maps $\ext(D,S)$ and $\con(S,D)$  (see also, for instance, \cite[Proposition 1.35, Lemma 1.36, Example 1.37, Proposition 2.11(1), Proposition 2.13(1), Proposition 2.15]{Pi}).

\begin{lemma}

\begin{enumerate}
\item The map $\ext $ is order-preserving, i.e., $\star_1 \leq \star_2$ implies  $(\star_1)^{\e} \leq (\star_2)^{\e}$.
\item The map $\ext $ preserves semistar operations of finite type, i.e., $\ext \mid_{\tiny\SStarf(D)} : \SStarf(D) \rightarrow \SStarf(S)$.
\item The map $\con$ is order-preserving, i.e., $\star_1 \leq \star_2$ implies  $(\star_1)^{\co} \leq (\star_2)^{\co} $.
\item The map $\con$  preserves semistar operations of finite type, i.e., $\con\mid_{\tiny\SStarf(S)}: \SStarf(S) \rightarrow \SStarf(D)$.

\item Let $\dd_D$ (respectively, $\dd_S$) the identity semistar operation on $D$ (respectively, on $S$), then  $(\dd_D)^{\e} = (\wS)^{\e} = \dd_S $.
\item $(\dd_S)^{\co} = \wS  $.

\item For each $\star \in{\mbox{\rm \SStar}}(S)$, $(\star^{\co})^{\e}Ê=  \star $, (i.e., $ \ext  \circ  \con  = \id_{\mbox{\tiny{\rm\SStar}(S)}}$).

\item For each $\star \in{\mbox{\rm \SStar}}(D)$, $ (\star^{\e})^{\co}Ê\geq \star$ (for short, we summarize   this property  by writing $ \con  \circ   \ext  \geq  \id_{\mbox{\tiny{\rm\SStar}(D)}}$). In particular, if $D \subsetneq S$, $\dd_D \lneq ((\dd_D)^{\e})^{\co}= (\dd_S)^{\co}= \wS$.

\end{enumerate}
\end{lemma}

\begin{remark}
{\rm
In relation with statements (2)  and (4) of the previous lemma, we observe that $\ext $ preserves stable semistar operations  and, if $D^\star = S$, then $\star$ is spectral on $D$ implies that $(\star)^{\e}$ is spectral on $S$ \cite[Proposition 2.11 (2) and (6)]{Pi}.
On the other hand,  $\con$  preserves neither stability nor spectrality. For instance,   $\dd_S$ is obviously spectral and hence stable on $S$ while, if $S$ is not a $D$-flat overring of $D$, $(\dd_S)^{\co} = \wS  $ is not stable (and, a fortiori, is not spectral) on $D$.
}
\end{remark}

The overring $S$ of $D$ is a {\it sublocalization  of } $D$ if $S$ is a nonempty intersection of ring of fractions
of $D$.  Thus $S$ is a sublocalization of $D$ if and only if there exists a nonempty family $ \{T_{\alpha} \mid \alpha \in\mathscr{A} \} $ of
multiplicatively closed subsets of nonzero elements of $D$ such that $S = \bigcap \{D_{T_\alpha} \mid \alpha \in \mathscr{A}  \}$.
 It is
well known that a sublocalization $S$ of $D$ is an intersection of localizations of $D$
at prime ideals, since each ring of fractions of $D$ is  an intersection of localizations of $D$  (see \cite{Gi-He} and \cite{Ri}).  Indeed
If $T$ is a multiplicatively closed subset of an integral domain $D$,  with $0\not\in T$, then
$D_T = \bigcap \{D_P
\mid  P \in \Spec(D) \mbox{ and }  P \cap  T  = \emptyset\}$.
Therefore, if $\{T_{\alpha} \mid \alpha \in\mathscr{A} \} $  is a family of multiplicatively closed sets of nonzero elements of $D$
and $S = \bigcap \{D_{T_\alpha} \mid \alpha \in \mathscr{A}  \}$, then
$S  = \bigcap\{D_P  \mid  P \in \Spec(D), P \cap  T_\alpha  = \emptyset, \mbox{ for some } \alpha \in \mathscr{A} \}$.

From the previous remarks, we deduce immediately:

\begin{lemma} \label{sublocalization}
Let $S$ be an overring of $D$. Then,
$S$ is a sublocalization of $D$ if and only if
$S = \bigcap\{D_P  \mid  P \in \Spec(D), \ S \subseteq D_P \}$.
\end{lemma}

Recall that,  by \cite[Theorem 1]{Ri},  $S$ is a $D$-flat overring of $D$ if and only if, for each $P \in \Spec(D)$,  either   $PS =S$ or $S \subseteq D_P$. Therefore, by Lemma \ref{sublocalization}, if $S$ is a $D$-flat overring of $D$ then $S$ is a sublocalization of $D$. However, the converse is not true  \cite[Section 2, Discussion 2.1]{HR}.

\begin{proposition}\label{Lemma A}
 Let $D$ be a domain and let $S =\bigcap \{D_{T_\alpha} \mid \alpha \in \mathscr{A}  \}$ be  a sub\-localization of $D$, where $ \{T_{\alpha} \mid \alpha \in\mathscr{A} \} $ is a given family of
multiplicatively closed subsets of nonzero elements of $D$.
Set $\boldsymbol{\mathcal{T}}:=
\boldsymbol{\mathcal{T}}(S) :=
 \{D_{T_{\alpha}} \mid \alpha \in \mathscr{A}  \} $, considered as a family of overrings of $S$,
 let
$\ast:= \wedge_{\boldsymbol{\mathcal{T}}(S) } \in \SStar(S)$ and set $\star_S := \ast^{\co} = \ast \circ \wedge_{\{S\}} \in \SStar(D)$, i.e. $E^{\star_S} :=
(ES)^\ast = \bigcap\{ED_{T_\alpha} \mid \alpha \in \mathscr{A}  \}$, for each $E \in \FF(D)$.
\begin{enumerate}
\item 
 $\star_S   = \wedge_{\boldsymbol{\mathcal{T}}(D)}$, where in the last equality the family $\boldsymbol{\mathcal{T}}(D)$ is the family $\boldsymbol{\mathcal{T}}$ considered as a family of overrings of $D$.
\item  
 $\widetilde{\wedge_{\{S\}}} \leq \wedge_{\{S\}} \leq \star_S 
$. 
\item
Let $  A_{1}, A_{2}, \dots , A_{n}  \in \FF(D)$. Then,
$(A_{1}\cap A_{2}\cap \dots \cap A_{n})^{\star_S} = ((A_{1}\cap A_{2}\cap \dots \cap A_{n})S)^{\ast}=(A_{1}S)^{\ast }\cap (A_{2}S)^{\ast }	\cap \dots \cap (A_{n}S)^{\ast } = A_{1}^{\star_S }\cap A_{2}^{\star_S}\cap \dots \cap A_{n}^{\star_S }$, i.e., $\star_S$ is a stable semistar operation on $D$.

\item If $F\in \f(D)$,  then $(FS)^{-1}=(F^{-1}S)^{\ast } = (F^{-1})^{\star_S}$.

\item Let $\vv(S)$ be the (semi)star $\vv$-operation of $S$ (i.e., $E^{\,\vvs(S)} := (S:(S:E))$ for each $E \in \FF(D)$).
If $F\in \f(D)$, then $
(FS)^{-1}=(F^{-1}S)^{\vvs(S)}.$
\end{enumerate}
\end{proposition}


\begin{proof}
(1) is a straightforward consequence of the definitions.

(2) In general, for each semistar operation $\star$ on $D$, the stable semistar operation of finite type $\widetilde{\star}$ is such that
$\widetilde{\star} \leq \star $.  
 The second inequality follows by observing that $ES  \subseteq (ES)^{\ast}$, for each $E \in \FF(D)$.

%

 (3)  follows easily from the fact that ${\star_S}$ coincides with $(\wedge_{\boldsymbol{\mathcal{T}}(S)})^{\co}$ and $\boldsymbol{\mathcal{T}}(S)$ is a family of   overrings of fractions of $S$ (and $D$), hence $\ast = \wedge_{\boldsymbol{\mathcal{T}}(S)}$ (respectively,   $(\wedge_{\boldsymbol{\mathcal{T}}(S)})^{\co}$) is a stable semistar operation on $S$ (respectively, on $D$).

(4) Let $F = (f_1, f_2, \dots, f_r)$, then $(FS)^{-1} = (S:FS) = \bigcap\{f_i^{-1}S\mid 1 \leq i \leq r\}$ and $F^{-1} = (D:F) = \bigcap\{f_i^{-1}D\mid 1 \leq i \leq r\}$. Therefore,  $(F^{-1})^{\star_S}=( \bigcap\{f_i^{-1}D\mid 1 \leq i \leq r\})^{\star_S} =
 \bigcap\{(f_i^{-1}D)^{\star_S}\mid 1 \leq i \leq r\}
= \bigcap\{(f_i^{-1}S)^\ast \mid 1 \leq i \leq r\}= \bigcap\{f_i^{-1}S \mid 1 \leq i \leq r\}= (FS)^{-1}$. Thus  $(F^{-1}S)^{\ast } = (F^{-1})^{\star_S}=(FS)^{-1}$.

 (5)  
Since $\ast$ is a (semi)star operation of $S$, it is clear that $\ast \leq \vv(S)$ (see  \cite[Theorem 34.1(4)]{Gil} and \cite[Lemma 1.11]{Pi}). By (4), we have $(FS)^{-1} = (F^{-1}S)^\ast$.
Therefore, $((FS)^{-1})^{\vvs(S)} = (FS)^{-1} =  (F^{-1}S)^\ast =  ((F^{-1}S)^\ast)^{\vvs(S)}= (F^{-1}S)^{\vvs(S)}$.
\end{proof}

\begin{remark}{\rm
As a straightforward consequence of the previous proposition, we re-obtain the following well known  properties.
If $S$ is a $D$-flat overring of $D$, then
\begin{enumerate}
\item for $A_{1}, A_{2}, \dots, A_{n} \in \FF(D)$,
$(A_{1}\cap A_{2}\cap\dots \cap A_{n})S = A_{1}S\cap A_{2}S\cap  \dots \cap A_{n}S$;
\item for each $F\in  \f(D)$,  $ (FS)^{-1} =F^{-1}S $ and  $ (FS)^{\vvs(S)} =(F^{\vvs(D)}S)^{\vvs(S)} $, 
 where $\vv(S)$ (respectively, $\vv(D)$) is  the (semi)star $\vv$-operation of $S$ (respectively, of $D$), for details see  \cite[Proposition 0.6(b)]{FG}.
\end{enumerate}
}
\end{remark}

\smallskip

Let $\star $ be a semistar operation on  the   integral domain $D$.
For $E\in \boldsymbol{\overline{F}}(D)$, we say that $E$ is  \it
$\star $-finite  \rm if there exists a $F\in \boldsymbol{f}(D)$
such that $F^{\star }=E^{\star }$. (Note that in the  above
definition, we do not require that $F\subseteq E$.) It is
immediate to see that if $\star _{1}\leq \star _{2}$ are semistar
operations and $E$ is $\star _{1}$-finite, then $E$ is $\star
_{2}$-finite.
In particular, if $E$ is $\star
_{\!_{f}}$-finite, then it is $ \star $-finite. The converse is
not true in general  (\cite[Remark 2.4]{FP}), and one can prove that $E$ is $\star
_{\!_{f}}$--finite if and only if there exists
 $F\in \boldsymbol{f}(D)$, $F\subseteq E$, such that $F^{\star
}=E^{\star }$ \cite[Lemma 2.3]{FP}. This result was proved in the star
operation setting by M. Zafrullah in \cite[Theorem 1.1]{Zaf 7}.

\begin{lemma}\label{Lemma E} Let $S$ be an overring of $D$ and $\ast$ a semistar operation on $S$.
Consider the semistar operation $\star_S:=\ast^{\co} $ on $D$. Let $I$ be a
nonzero ideal of $D$ and assume that $I^{\star_S} := (IS)^{\ast }=((x_{1}, x_2, \dots ,x_{n})S)^{\ast }$, where
$x_{k}\in IS$, for $1 \leq k \leq n$. Then,  we can find  a finitely generated ideal $J$ of $D$, with $J \subseteq I$, such that,
$$
I^{\star_S}= (IS)^{\ast }=(JS)^{\ast }= J^{\star_S}\,.$$
\end{lemma}

\begin{proof}
Indeed,  as $x_{k}$ $\in IS$, we have
 $x_{k}=
\sum_{j=1}^{r_{k}}i_{kj}s_{j}$, where $i_{kj}\in I$ and $s_{j}\in S$.
 Then  
$x_{k}S\subseteq I_kS$ for some finitely generated ideal $I_k\subseteq I$ of $D$, for every $k$.  Take $J:=\sum_k I_k$ and the verification of the claim is straightforward.
\end{proof}


\begin{proposition}\label{Proposition G}
Let $\{S_{\lambda}\mid  \lambda \in \Lambda\}$ be a family of overrings of $D$ and let 
$\ast _{\lambda}$ be a (semi)star operation on $S_\lambda$. 
 Set $\star_{S_{\lambda}} := (\ast_\lambda)^{\co}$,
 i.e., $E^{\star_{S_{\lambda}}} :=
(ES_\lambda)^{\ast_{\lambda}}$, for each $E \in \FF(D)$. Consider the semistar operation on $D$, 
$\boldsymbol{\star}:= \bigwedge \star_{S_{\lambda}}: \FF(D) \rightarrow \FF(D)$, defined by $E \mapsto \bigcap \{ES_\lambda \mid  \lambda \in \Lambda  \}$.

Suppose that $D= \bigcap \{S_\lambda \mid \lambda \in \Lambda\}$ is locally finite.  If $I$ is a nonzero ideal of $D$ such that $(IS_{\lambda})^{\ast _{\lambda}} =((x_{\lambda1}, x_{\lambda2}, \dots ,x_{\lambda n_{\lambda}})S_{\lambda})^{\ast _{\lambda}}$
with $x_{\lambda\mu}\in IS_{\lambda}$, for each $\lambda \in \Lambda$ and  $1 \leq \mu \leq n_\lambda$,
 then there
is a finitely generated ideal $J\subseteq I$ in $D$ such that $J^{\boldsymbol{\star}}=I^{\boldsymbol{\star}}$.

\end{proposition}
\begin{proof}

Since $D= \bigcap \{S_\lambda \mid \lambda \in \Lambda\}$ is locally finite, we have $
IS_{k}\neq S_{k}$ for at most a finite subset $\{ S_k \mid  1 \leq k\leq n \}$.
Now, take  a nonzero element $j\in I$, for the same reason,  $j$ is a nonunit in  only finitely many  overrings $S_\lambda$ and, by the previous considerations, we can assume that $j$ is a nonunit in
$\{S_{1},S_{2}, \dots, S_{n}, S_{n+1}, \dots, S_m \mid  \mbox{ for some } m \geq n \}$.

If $m=n$, then $jD \subseteq I$ is such that $jS_{k}\neq S_{k},$ precisely for $1 \leq k\leq n$.

 If $m \gneq n$, since $IS_{h}=S_{h}$ for each $n+1 \leq h \leq m$, there exist a finitely generated ideal $I_h \subseteq I$   such that $I_hS_{h}=S_{h}$.
Thus, the finitely generated ideal $J_0:=jD+\sum_h I_h\subseteq I$ ensures
that $J_0S_{k}\neq S_{k},$ precisely for $1 \leq k\leq n$.

From Lemma \ref{Lemma E}, for each $\lambda \in \Lambda$, we know that $ I^{\star_{S_\lambda}}= (IS_\lambda)^{\ast_\lambda }=(J_\lambda S_\lambda)^{\ast_\lambda }=(J_\lambda)^{\star_{S_\lambda}}$, for some finitely generated ideal $J_\lambda \subseteq I$ of $D$.
In particular, if we consider the finite subset $\{ S_k \mid  1 \leq k\leq n \}$ of $ \{ S_\lambda \mid  \lambda \in \Lambda \}$, then we can find a finite set of finitely generated ideals $\{ J_k \mid  1 \leq k\leq n \}$  contained in $I$ such that
 $(IS_k)^{\ast_k }=(J_k S_k)^{\ast_k}$, where ${\ast_k} := {\ast_{\lambda_k}}$, for $ 1 \leq k\leq n$.
 Set $J:= J_0 + J_1+\dots + J_n$, by construction, it is easy to see that $J$ is finitely generated ideal of $D$ contained in $I$.

Therefore,  $(IS_\lambda)^{\ast_\lambda} =  S_\lambda =  (JS_\lambda)^{\ast_\lambda} $, for each $\lambda \in \Lambda \setminus \{1, 2, \dots, n\}$ (since, in the present situation, $J_0S_\lambda =JS_\lambda =IS_\lambda =  S_\lambda $).
 For $k\in \{1, 2, \dots,n\}$, we
have $(IS_{k})^{\ast _{k}}=(J_kS_{k})^{\ast
_{k}}\subseteq (JS_{k})^{\ast _{k}}$.
Thus, for all $\lambda\in \Lambda$, we have $(IS_{\lambda})^{\ast _{\lambda}}\subseteq
(JS_{\lambda})^{\ast _{\lambda}}$ and so we conclude that  $I^{\boldsymbol{\star} }\subseteq J^{\boldsymbol{\star} }$.
 The opposite inclusion is trivial, since $J \subseteq I$.
\end{proof}

\medskip


\section{Sublocalizations and Pr\"ufer $\mbox{$\vv$}$-multiplication domains}

Let $\star$ be a semistar operation on an integral domain $D$.
For   a nonzero ideal $I$ of $D$, we say that $I$ is \it{$\star
$--invertible} \rm if $(II^{-1})^{\star }=D^{\star }$. From the
fact that $\QMax^{\widetilde{ \star }}(D)=\QMax^{\star
_{\!_{f}}}(D)$, it easily follows that an ideal $I$ is
$\widetilde{\star }$--invertible if and only if $I$ is $\star
_{_{\!f}}$--invertible (note that if $\boldsymbol{\star} $ is a semistar
operation of finite type, then $(II^{-1})^{\boldsymbol{\star} }=D^{\boldsymbol{\star}  }$ if
and only if $ II^{-1}\not\subseteq M$ for all $M\in
\text{QMax}^{\boldsymbol{\star} }(D)$). It is   well known   that if $I$ is $\star
_{_{\!f}}$--invertible, then $I$ and $I^{-1}$ are both $\star
_{_{\!f}}$--finite \cite[Proposition 2.6]{FP}.

\medskip

An integral domain $D$ is called a \it  Pr\"{u}fer $\star
$--multiplication domain  \rm (for short,  \it P$\star $MD\rm) if
every nonzero finitely generated ideal of  $D$ is $\star
_{_{\!f}}$--invertible (cf. for instance \cite{FJS}). Note that
    for $\star = \ast$   a   star operation of finite type on $D$,
P$\ast$MD's were intro\-duced by Houston, Malik, and Mott in
\cite{[HMM]} as $\ast$--multiplication domains.  When $\star = \vv$, we have the classical notion
of P$\vv$MD (cf. \!for in\-stance \cite{Gri 1}, \cite{MZ} and \cite{Kan});
when $\star =\dd$, where $\dd$ denotes the identity (semi)star
operation, we have the notion of Pr\"{u}fer domain \cite[Theorem
22.1]{Gil}.  For star operations $\ast$, the only P$\ast$MDs are the P$v$MDs and the Pr\"ufer domains since in
a P$\ast$MD, $\ast_f=t$   (see \cite[Theorem 3.5]{Kan} and \cite[Proposition 3.4]{FJS}). 

  Note that from the definition and from the previous
observations, it immediately follows that the notions of P$\star
$MD, P$\star _{_{\!f}}$MD,  and P$ \widetilde{\star }$MD coincide.
\medskip

As in the star case \cite[Corollary 2.10]{And-Cook}, it is  well known that, for each semistar operation $\star$, we have $\stt = \wedge_{\QMax^{\stf}(D)}$, i.e., for each $E \in \FF(D)$,
$$
E^{\stt} = \bigcap \{ ED_Q \mid Q \in\QMax^{\stf}(D) \}\,.
$$
From this fact, it can be deduced that  $D$ is a  P$\star
$MD if and only if $D_Q$ is a valuation domain for each $Q \in \QMax^{\stf}(D)$ \cite[Theorem 3.1]{FJS}.

Recall that an {\it essential valuation overring} $V$ of an integral domain $D$ is a valuation overring of $D$ such that $V = D_P$ for some $P \in \Spec(D)$; in this situation, $P$ is called {\it essential prime}.
 A family of overrings $ \{S_\lambda \mid \lambda\in \Lambda  \}$ of $D$ is said   an {\it essential representation} (or, an {\it  essential defining family}) {\it of}
$D$, if $D = \bigcap \{S_\lambda \mid \lambda\in \Lambda  \}$
and each $S_\lambda$  is an essential valuation overring of $D$. An {\it essential domain} is an integral domain having an essential representation.
A P$\vv$MD is always essential because $D_Q$ is a valuation domain
for each $Q \in\QMax^{\ttts}(D)$ \cite[Theorem 3.2]{Kan}.

\begin{theorem}\label{Theorem H}
Let $\{S_{\lambda}\mid  \lambda \in \Lambda\}$ be a family of sublocalizations
of an integral domain $D$.
Suppose that
 $D=\bigcap \{S_{\lambda}\mid  \lambda \in \Lambda\}$ where the intersection is
locally finite.
 \begin{enumerate}
 \item[\bf(1)]
Let $\boldsymbol{\mathcal{T}}_{\!\lambda}:=
\boldsymbol{\mathcal{T}}(S_\lambda):= \{D_{T_{\alpha_{\lambda}}} \mid \alpha_{\lambda}\in \mathscr{A}_{\lambda}\}$ be a defining family of $S_{\lambda}$ and let $
\ast _{\lambda}$ be the (semi)star operation on $S_\lambda$ induced  by $\boldsymbol{\mathcal{T}}_\lambda$, i.e.,
$
\ast _{\lambda} := \wedge_{\boldsymbol{\mathcal{T}}_{\lambda}}$.
As in Proposition \ref{Proposition G}, set $\boldsymbol{\star}:= \bigwedge \{(\ast_\lambda)^{\co} \mid \lambda \in \Lambda\}$.\\
Assume that, for each $\lambda \in \Lambda$,
 \begin{enumerate}
\item[(a)]
 $\ast _{\lambda}$ is a (semi)star operation of finite type of $S_\lambda$,  \; \; and
\item[(b)] $S_{\lambda}$ is a P\! $\ast _{\lambda}$MD, 
\end{enumerate}
 then $D
$ is a P$\boldsymbol{\star}$MD, and so $D$ is a P\! $\vv$MD.

 \item[\bf(2)] Assume that each of $S_{\lambda}$ is a P\! $\vv$MD, 
then $D
$ is a P\! $\vv$MD.
\end{enumerate}
\end{theorem}
\begin{proof}
(1)
Recall that, given a semistar operation $\star$ on an integral domain $D$, $D$ is a P$\star$MD if and only if $\widetilde{\star} $ is a \texttt{eab} semistar operation, i.e., $\widetilde{\star} =\sta$ \cite[Theorem 3.1]{FJS}.

We start by observing that, in the present situation,  $\ast _{\lambda}$ is a stable (semi)star operation, because the family $\boldsymbol{\mathcal{T}}_\lambda$ consists of  rings of fractions. 
Since we are assuming that $\ast _{\lambda}$ is 
of finite type of $S_\lambda$ and  $S_{\lambda}$ is a P$\ast _{\lambda}$MD, we have 
 $\ast _{\lambda} = \widetilde{\ast _{\lambda}} = (\ast _{\lambda})_a$ \cite[Theorem 3.1]{FJS}.

Moreover, $\boldsymbol{\star} = \bigwedge \{ (\ast _{\lambda})^{c_\lambda}\mid \lambda \in \Lambda \}$  is a (semi)star operation of finite type of $D$, since the intersection $D=\bigcap \{S_{\lambda}\mid  \lambda \in \Lambda\}$   is
locally finite and each $\ast _{\lambda}$ is of finite type (see \cite[Theorem 2(4)]{And} and \cite[Corollary 2.9]{FiSp}).
Clearly,  $\boldsymbol{\star}$ is a stable  and valuative (semi)star operation on $D$,
because, in the present setting,  each $(\ast _{\lambda})^{c_\lambda} \ (= \con(D, S_\lambda)(\ast _{\lambda}))$ is a stable (since it is induced by a family of rings of fractions of $D$)
and valuative (semi)star operation on  $D$ (since  $\ast _{\lambda}$ is valutative and the valuation overrings of  $S_{\lambda}$  are valuation overrings of $D$).
We conclude that  $\widetilde{\boldsymbol{\star}} = \boldsymbol{\star}$ is   an \texttt{eab} (semi)star operation and so $D$ is a P${\boldsymbol{\star}}$MD. Since ${\boldsymbol{\star}} \leq \vv$, then  $D$ is  also a P$\vv$MD.

 (2)  For each  $\lambda \in \Lambda$, we take as defining family of $S_{\lambda}$ the family of valuation overrings  $\boldsymbol{\mathcal{T}}_\lambda:=
\boldsymbol{\mathcal{T}}(S_\lambda):=
 \{(S_\lambda)_{\mathfrak{q}_{\lambda}} \mid \mathfrak{q}_{\lambda}\in \mathscr{A}_{\lambda} := \Max^{\ttts}(S_{\lambda})\}$.
 In the present situation, the (semi)star operation on $S_\lambda$ associated to  $\boldsymbol{\mathcal{T}}_\lambda$, i.e., $\ast_\lambda = \wedge_{\boldsymbol{\mathcal{T}}_\lambda}$, coincides with $\ww_\lambda$, that is the $\ww$-operation on
 $S_\lambda$. It is easy to see that the assumptions of (1) are satisfied  (after recalling that a P$\vv$MD coincides with a P$\ww$MD) and so, if we denote by $\boldsymbol{\star}$ the (semi)star operation
 $\bigwedge \{ (\ww_\lambda)^{c_\lambda} \mid \lambda \in \Lambda\}$, we can conclude by (1) that $D$ is a P$\boldsymbol{\star}$MD. In particular, since $\boldsymbol{\star} \leq \vv$, $D$ is a P$\vv$MD.
 %
 %
%

\end{proof}

Recall that an overring $S$ of an integral domain $D$ is a {\it $\ttt$-flat overring of $D$} if, for
each maximal $\ttt$-ideal $M$ of $S$,   $S_{M}=D_{(M\cap D)}$ \cite{KP}.

\begin{remark}
{\rm 
Note that it is possible to give a direct and independent proof of Theorem \ref {Theorem H}(2) under the assumptions that 
 $D=\bigcap \{S_{\lambda}\mid  \lambda \in \Lambda\}$, the intersection is
locally finite and each $S_{\lambda}$ is $t$-flat.

By assumption,
$$D=\bigcap \{S_{\lambda}\mid  \lambda \in \Lambda\}
= \bigcap \left\{ \bigcap \{(S_{\lambda})_{\mathfrak{q}_{\lambda}}\mid {\mathfrak{q}_{\lambda}} \in \Max^{\ttts}(S_{\lambda}\} \mid  \lambda \in \Lambda\right\} $$
and  the valuation overring $(S_{\lambda})_{\mathfrak{q}_{\lambda}}$  is essential for $D$, for each $\lambda \in \Lambda$ and for each ${\mathfrak{q}_{\lambda}} \in \Max^{\ttts}(S_{\lambda})$.
Now, by Lemma 8 of \cite{Zaf 2}, an essential domain  $D$ is
a P$\vv$MD if and only if, for every pair of elements $a,b\in D\backslash \{0\}$, the ideal $
aD\cap bD$ is a $\vv$-ideal of finite type.

As each $S_{\lambda}$ is a P$\vv$MD,  $aS_{\lambda}\cap
bS_{\lambda}$  is a $\vv_{\lambda}$-ideal of finite type, thus we can find $x_{\lambda1}, x_{\lambda 2},\dots ,x_{\lambda n_{\lambda}} \in S_\lambda$ such that
$$
aS_{\lambda}\cap bS_{\lambda}  =
(x_{\lambda1}, x_{\lambda 2},\dots , x_{\lambda n_{\lambda}})^{\vvs_{\lambda}} =
(x_{\lambda1}, x_{\lambda 2},\dots , x_{\lambda n_{\lambda}})^{\ttts_{\lambda}} =
(x_{\lambda1}, x_{\lambda 2},\dots , x_{\lambda n_{\lambda}})^{\wws_{\lambda}}
$$
where $\vv_{\lambda}$, $\ttt_{\lambda}$, and $\ww_{\lambda}$  are the $\vv$-, the $\ttt$-, and $\ww$-operation on  the  $S_{\lambda}$'s, and we already observed that the $\ttt$-, and $\ww$-operation coincide on the P$\vv$MD $S_\lambda$ \cite[Theorem 3.5]{Kan} (via \cite[Theorem 4.7]{Zaf 10}).

On the other hand, by
Proposition \ref{Lemma A}(3), $((aD\cap bD)S_{\lambda})^{\wws_{\lambda}}=
(aS_{\lambda})^{\wws_{\lambda}} \cap (bS_{\lambda})^{\wws_{\lambda}} =aS_{\lambda}\cap bS_{\lambda}$, for each $\lambda$.
Therefore,
$$(aS_{\lambda}\cap bS_{\lambda})^{\wws_{\lambda}} = (x_{\lambda1}, x_{\lambda 2},\dots , x_{\lambda n_{\lambda}})^{\wws_{\lambda}}
$$
and, necessarily, $x_{\lambda k}\in aS_{\lambda}\cap bS_{\lambda}$, for  $1 \leq k \leq n_\lambda$.
 Let $I := aD\cap bD$,  by Lemma \ref{Lemma E}, we can find  a finitely generated ideal $J_\lambda:= (j_{\lambda 1}, j_{\lambda 2}, \dots, j_{\lambda r_\lambda})D$,   with $J_\lambda \subseteq I$, such that,
$$
 (IS_{\lambda})^{\wws_{\lambda} }=((j_{\lambda 1}, j_{\lambda 2}, \dots, j_{\lambda r_\lambda})S_{\lambda})^{\wws_{\lambda} }\, .$$
Next, as $D=\bigcap \{S_{\lambda}\mid  \lambda \in \Lambda\}$  is of finite character
then, in particular, $\boldsymbol{\star} $ is a (semi)star operation on $D$.  Moreover, by
 Proposition \ref{Proposition G},  there exists a finitely generated ideal $J$ of $D$ , with $J \subseteq I$, such that
$J^{\boldsymbol{\star}} = I^{\boldsymbol{\star}}$. Since $\boldsymbol{\star} \leq \vv$,
 then $J^{\vvs}= (J^{\boldsymbol{\star}})^{\vvs} =(I^{\boldsymbol{\star}})^{\vvs}= I^{\vvs} = aD \cap bD$.
 }
\end{remark}

From the previous Theorem \ref{Theorem H}, we deduce immediately  the following two corollaries.

\begin{corollary}\label{Corollary L}
Let $\{S_{\lambda}\mid  \lambda \in \Lambda\}$ be a family of essential valuation
overrings of $D$ such that $D=\bigcap \{S_{\lambda}\mid  \lambda \in \Lambda\}$, where the
intersection is locally finite. Then $D$ is a P \!$\vv$MD.
\end{corollary}

\begin{corollary}\label{Corollary M}
 Let $\{S_{\lambda}\mid  \lambda \in \Lambda\}$ be a family of sublocalizations
of an integral domain $D$.
Assume that
 $D=\bigcap \{S_{\lambda}\mid  \lambda \in \Lambda\}$, where the intersection is
locally finite.
 If each of $S_\lambda$ is a Pr\"ufer domain,
then $D$ is a P \!$\vv$MD.
\end{corollary}

\medskip

Let $\star$ be a semistar operation on an integral domain $D$. We say
that $D$ is a {\it $\star$-Noetherian domain} if $D$ has the ascending chain condition on
quasi-$\star$-ideals.
Note that the $\dd$-Noetherian domains are just the usual Noetherian domains and the notions
of $\vv$-Noetherian (respectively, $\ww$-Noetherian) domain and {\it Mori} (respectively,
{\it strong Mori}) {\it domain} coincide.
Recall that, in the star case, the concept of {\it star Noetherian domain} has been
introduced by M. Zafrullah \cite{Zaf 7} (see, also, for instance, \cite{And-And}, \cite{G-J-S} and \cite{Pi}).

The following properties follow easily from the definitions (for more details, see for instance \cite[Lemma 4.16 and 4.18, and Corollary 4.19]{Pi}  or \cite[Lemma  3.1 and 3.3]{EFP}).
\begin{enumerate}
\item If $\star_1 \leq  \star_2$ are two semistar operations on $D$, then $D$ is $\star_1$-Noetherian implies that
$D$ is $\star_2$-Noetherian; in particular, a Noetherian domain is a $\star$-Noetherian domain, for any semistar
operation $\star$ on $D$.

\item  If $\star$ is a (semi)star operation and if $D$ is a $\star$-Noetherian domain,
then $D$ is a Mori domain.

\item $D$ is $\star$-Noetherian if and only if, for each nonzero ideal $I$ of $D$, there exists a nonzero finitely generated ideal $J$ of $D$ such that $J \subseteq I$ and $J^\star = I^\star$.

\item $D$ is $\star$-Noetherian if and only if $D$   is  $\stf$-Noetherian; in particular, the notions of $\vv$-Noetherian domain and $\ttt$-Noetherian domain coincide with the notion of Mori domain.

\end{enumerate}

\bigskip
 The following Proposition extends to the semistar setting a  result obtained by  Querr\'e   in 1976 (see the following Remark \ref{mori-case}).

\begin{proposition}\label{Proposition PG}
Let $\{S_{\lambda}\mid  \lambda \in \Lambda\}$ be a family of overrings of $D$ such that
 $D=\bigcap \{S_{\lambda}\mid  \lambda \in \Lambda\}$ and the intersection is
locally finite.
Let  $\ast _{\lambda}$ be a (semi)star operation on $S_\lambda$ and consider the semistar operation on $D$, $\boldsymbol{\star}:= \bigwedge \{(\ast_\lambda)^{{\co}_{\lambda}} \mid \lambda \in \Lambda\}$.
\begin{itemize}
\item [(1)] Assume that, for each $\lambda \in \Lambda$,
 $S_{\lambda}$ is
$\ast _{\lambda}$-Noetherian, then $D$ is $\boldsymbol{\star}$-Noetherian.
\item[(2)] Assume that, for each $\lambda \in \Lambda$,
 $S_{\lambda}$ is
 ${\widetilde{\ast_\lambda}}$-Noetherian and   that the semistar operation  $\bullet:= \bigwedge \{(\widetilde{\ast_\lambda})^{{\co}_{\lambda}} \mid \lambda \in \Lambda\}$ on $D$ is stable (e.g., when the 
$S_{\lambda}$'s are quotient rings of $D$), then $D$ is  $\widetilde{\boldsymbol{\star}}$-Noetherian.
\end{itemize}
\end{proposition}

\begin{proof}
 (1)  Given a nonzero ideal $I$ of $D$, since  $S_{\lambda}$ is ${\ast_\lambda}$-Noetherian
there exists a nonzero finitely generated ideal  $J_\lambda$ in $S_\lambda$ such that $J_{\lambda}\subseteq IS_{\lambda}$ and  $(IS_{\lambda})^{{\ast_\lambda}}= (J_\lambda)^{{\ast_\lambda}}$.  Since $D=\bigcap \{S_{\lambda}\mid  \lambda \in \Lambda\}$ and the intersection is
locally finite,
by Lemma \ref{Lemma E} and Proposition \ref{Proposition G}, we can  assume that $(J_\lambda)^{\ast_\lambda} = (JS_\lambda)^{\ast_\lambda}$,   for each $\lambda \in \Lambda$,  where $J$ is a finitely generated ideal of $D$   such that  $J \subseteq I$ and $J^{\boldsymbol{\star}}= I^{\boldsymbol{\star}}$.

(2)  Note that $\bullet \leq \widetilde{\boldsymbol{\star}}\leq \boldsymbol{\star}_{\!_f} \leq \boldsymbol{\star} $. Indeed, we have  $\bullet$ is stable; moreover $\bullet$ is a (semi)star operation of finite type, since $(\widetilde{\ast_\lambda})^{{\co}_{\lambda}}                                                                                                                                                                                                                                                                                                                                                                                                                                                                                                                                                                                                                             $ is of finite type, for each  $\lambda \in \Lambda$, and $D=\bigcap \{S_{\lambda}\mid  \lambda \in \Lambda\}$  is
locally finite (see \cite[Theorem 2(4)]{And} and \cite[Proposition 2.9]{FiSp}).

If we show that $D$ is  $\bullet $-Noetherian, then a fortiori we have that $D$ is  $\widetilde{\boldsymbol{\star}}$-Noetherian.
For this, given ideal $I$ of $D$ we have $(IS_{\lambda})^{\widetilde{\ast_\lambda}} =
\bigcap \{ (IS_{\lambda})_{\mathfrak{q}_{\lambda \mu}} \mid \mathfrak{q}_{\lambda \mu} \in \Max^{\widetilde{\ast_\lambda}}(S_\lambda) \}$.
Since  $S_{\lambda}$ is ${\widetilde{\ast_\lambda}}$-Noetherian   (and ${\widetilde{\ast_\lambda}} $ is of finite type), 
there exists a finitely generated ideal  $J_\lambda$ in $S_\lambda$ such that
$J_{\lambda}\subseteq IS_{\lambda}$ and $(J_\lambda)^{\widetilde{\ast_\lambda}} =(IS_{\lambda})^{\widetilde{\ast_\lambda}}
$.
Again, as $D=\bigcap \{S_{\lambda}\mid  \lambda \in \Lambda\}$  is
locally finite, Proposition \ref{Proposition G}
applies, and so there exists a finitely generated ideal $J$ of $D$, such that $J \subseteq I$ and $J^\bullet = I^\bullet$.  Therefore, $D$ is $\bullet$-Noetherian.
\end{proof}

 From the previous proposition, we easily deduce the following. 

\begin{corollary} \label{cor-PG}
 If $\bigcap \{S_{\lambda}\mid  \lambda \in \Lambda\}$ is a locally finite defining family of  overrings 
 (respectively, $\ttt$-flat overrings) of an integral domain $D$,
and if each of $S_{\lambda}$ is a Mori (respectively, strong Mori) domain, then  $D$ is a Mori (respectively, strong Mori) domain.

\end{corollary}


\begin{remark} \label{mori-case}
{\rm  Note that the   Mori domain case in Corollary \ref{cor-PG} can be viewed as a ``non completely integrally closed version'' of the following well known result  \cite[Proposition 1.4]{Fos}: If $\bigcap \{S_{\lambda}\mid  \lambda \in \Lambda\}$ is a locally finite defining family of an integral domain $D$ and  if each of $S_{\lambda}$ a Krull domain, then  $D$ is a Krull domain.

 Moreover, recall that N. Dessagnes in 1987 proved that the intersection of any locally finite family of Mori domains, all contained in the same integral domain, is a Mori domain \cite[Proposition 3.2]{de} (see also \cite[Th\'eor\`eme 1]{Q}).
}
\end{remark}

Recall that an integral domain $D$ is a {\it weakly Krull domain} if
$D=\bigcap \{ D_{P} \mid P\in X^{1}(D) \}$,  where $X^{1}(D)$ denotes the set of height one primes of $D$, and the intersection is locally
finite. Weakly Krull domains were studied in \cite{AMZ 2}.

 It is
well known that if $D$ is a Mori domain then so is each of its rings of
fractions  \cite[Th\'eor\`eme 2]{Q}. Using this piece of information and Proposition \ref{Proposition PG}, we deduce immediately the following.

\begin{corollary}\label{Corollary R} A weakly Krull domain $D$ is a Mori domain if and only if $
D_{P} $ is a Mori domain for each $P\in X^{1}(D).$
\end{corollary}

\section{Essential domains and Pr\"ufer $\vv$-multiplication domains}

In this section, we introduce a weak form of the finite character property of a defining family of a domain. As an application,  we shed new light on the question of when an essential domain is a P$\vv$MD solved recently  by Finocchiaro and Tartarone \cite{FT} using topological methods.

\medskip

Let $D$ be an integral domain, let
 $\mathcal{E}(D):=\{P\in \Spec(D) \mid D_P \, \mbox{ a valuation domain}\}$ be the set of all essential valuation overrings of $D$,
and let  $\emptyset \neq X\subseteq \Spec((D)$.
 We say that the domain $D$ is {\it $X$-essential }
 if $X\subseteq\mathcal{E}(D)$ and $D=\bigcap \{D_{P} \mid P \in X\}$.

Recall from \cite{BH} that a
prime ideal $Q$ of $D$ is an {\it associated prime of a principal
ideal $aD$ of $D$}, if $Q$ is minimal over $(aD :bD)$ for
some $b \in D \setminus  aD$ . For brevity, we call $Q$ an
{\it associated prime of $D$} and we denote by  $\Assp(D)$ the set of the associated prime ideals of $D$.
We say that $D$ is a {\it \texttt{P}-domain} if,
for every $Q\in  \Assp(D)$, $D_Q$ is a valuation domain  \cite{MZ}.  Note that a P$v$MD is a \texttt{P}-domain and not  conversely    \cite[Corollary 1.4 and Example 2.1]{MZ}.

As we   remarked above an important class of classical domains are $X$-essential for some nonempty set $X\subseteq \Spec(D)$,   i.e., weakly Krull domains, for $X= X^{1}(D)$. Moreover, 
if $X=\Max(D)$
 (or, even, $X = \Spec(D)$) (respectively, $X=\Max^{\ttts}(D)$;
  $X=\Assp(D)$) we get Pr\"ufer domains (respectively, P$\vv$MDs;  \texttt{P}-domains).

Let $D$ be an $X$-essential domain,  the (semi)star operation on $D$, $\ast_X$, induced by the nonempty family of overrings
$\boldsymbol{\mathscr{X}}:= \{D_{P} \mid P\in X \}$, i.e., $\ast_X := \wedge_ {\boldsymbol{\mathscr{X}}}$ (defined by $E^{\ast_X} := \bigcap \{ ED_P \mid P \in X \}$ for each $E \in  \FF(D)$),
 is crucial for studying these  domains as the following proposition shows.

\begin{proposition} \label{essential}
Let $D$ be an integral domain, let $\emptyset \neq X \subseteq \Spec(D)$ such that
$D=\bigcap \{D_{P} \mid P \in X\}$ and $\ast_X$ the star operation
on $D$ induced by the family of overrings $ \{D_{P} \mid P \in X\}$.
Then, the following are equivalent.
\begin{itemize}
\item[(i)] $D$ is an $X$-essential domain.
\item[(ii)] Every $\ast_X$-finite ideal is   $\ast_X$-invertible.
\end{itemize}
\end{proposition}
\begin{proof}
(i) $\Rightarrow$ (ii) Let $I\in \f(D)$ and $P\in X$.
Then $II^{-1}D_P=ID_P(ID_P)^{-1}=D_P$ since $D_P$ is a valuation domain. Hence, $(II^{-1})^{\ast_X}=D$.

(ii) $\Rightarrow$ (i)  Let $P\in X$ and $J$ a nonzero finitely generated ideal of $D_P$. Then $J=ID_P$ for
some finitely generated ideal $I$ of $D$. We have $JJ^{-1}=ID_P(ID_P)^{-1}=(II^{-1})D_P=(II^{-1})^{\ast_X}D_P=D_P$. So
$D_P$ is a local Pr\"ufer domain, and hence a valuation domain.
\end{proof}

\begin{remark}
{\rm
Note that Proposition \ref{essential} provides a general setting for a well known result on Pr\"ufer domains
(i.e., for $X=\Max(D)$ and $\ast_X=\dd$) or on  P$\vv$MDs (i.e., for $X=\Max^{\ttts}(D)$ and $\ast_X=\ww$).
On the other hand, a Dedekind domain (respectively, a Krull domain) is a Pr\"ufer Noetherian domain (respectively, a P$\vv$MD Mori ($\ww$-Noetherian) domain; note that in this situation $\ttt= \ww$), that is, in both cases,  $X$-essential and $\ast_X$-Noetherian domain.
We call a {\it $X$-Dedekind domain} an integral domain with these latter properties.
Since  on an $X$-essential and $\ast_X$-Noetherian domain, the (semi)star operation $\ast_X$ is of finite type, then  $ {\ast_X} \leq \ttt $ and so $\Max^{\ttts}(D)\subseteq  \Max^{\ast_X}(D)$ \cite[Lemma 2.1]{AMZ 2}.
Hence  an $X$-Dedekind domain $D$ is always a P$\vv$MD with $\ast_X=\ww$.
Thus, Dedekind  domains (i.e., when $\ast_X=\ww=\dd $) and Krull domains (i.e., when $\ast_X=\ww$) are   the only  $X$-Dedekind domains.
}
 \end{remark}

We say that  a defining family $\{S_\lambda \mid \lambda \in \Lambda \}$ of an integral domain $D$  has {\it GV-finite character property} if, for each  ideal $I$ of $D$ such that
 $IS_\lambda=S_\lambda$ for every $\lambda\in \Lambda $,
 there exists  a  finitely generated ideal  $J\subseteq I$ of $D$ such that $JS_\lambda=S_\lambda$ for every $\lambda$.
 Note that   the abbreviation``GV'' stands for Glaz-Vasconcelos, since we will see that the  GV-finite character property can be characterized by a general version of the notion of $H$-domain, introduced by  Glaz and Vasconcelos in \cite{G-V}.

\medskip

Obviously,  every defining family of overrings of a Noetherian domain has GV-finite character property.
Note that GV-finite character property is an extension of the finite character property. Indeed, assume that $D = \bigcap \{ S_\lambda \mid \lambda \in \Lambda \}$ has the finite character property and let $I$ be an ideal of $D$ such that $IS_\lambda=S_\lambda$ for every $\lambda$.
Let $0\not=x\in I$ and let $\{S_{\lambda_k} \mid 1 \leq k \leq n \}$ be the only $\lambda$'s such that $xS_\lambda \not=S_\lambda$.
For each $k$, there exists $J_k\subseteq I$ a finitely generated ideal of $D$
such that $IS_{\lambda_k}=J_kS_{\lambda_k}=S_{\lambda_k}$.
Take $J$ to be the finitely generated subideal of $I$ generated by $x$ and the $J_k$'s,  for $1 \leq k \leq n$, then it is straightforward that $JS_\lambda = IS_\lambda$ for each $\lambda \in \Lambda$.

\begin{proposition} \label{GVFCP} 
  Let $D$ be an integral domain and let $\boldsymbol{\mathscr{S}}:= \{S_\lambda \mid \lambda \in \Lambda\}$ be  a defining family of overrings of $D$.
 Denote by  $\ast$ the (semi)star operation
on $D$ induced by the defining family of overrings $\boldsymbol{\mathscr{S}}$ of $D$, i.e.,
 $\ast := \wedge_{\boldsymbol{\mathscr{S}}}$. Then, the following are equivalent.
\begin{itemize}
\item[(i)]  $\boldsymbol{\mathscr{S}}$  has GV-finite character property;
\item[(ii)] for every ideal $I$ of $D$ such that $I^\ast=D$, there exists a finitely generated  $J$ ideal of $D$ such that $J\subseteq I$  and $J^\ast=D$;
\item[(iii)] the stable (semi)star operation $\bar{\ast}$, canonically associated to $\ast$,  is of finite type, i.e., $\bar{\ast}=\widetilde{\ast}$.

\end{itemize}
\end{proposition}

\begin{proof}
(i) $\Leftrightarrow$ (ii)  is straightforward.

(ii) $\Rightarrow$ (iii) Let $E\in \FF(D)$ and $x\in E^{\bar{\ast}}$. Let $I$ be a nonzero ideal of $D$ such that
$xI\subseteq E$ with $I^\ast=D$. By assumption, we can take $I$ finitely generated. Let $ F:=xI \in \f(D)$. Then $F\subseteq E$ and $x\in  F^{\bar{\ast}}$. Thus $\bar{\ast}$ is of finite type.

(iii) $\Rightarrow$ (ii) is an easy consequence of the definitions.
\end{proof}

The case when $D$ has a defining family of quotient rings, that is $D=\bigcap\{D_{P} \mid  P\in X\}$ for
 some $X\sub\Spec(D)$ is of particular interest.
  In this case, if  the  defining family $\{D_{P} \mid  P\in X\}$ of $D$ has GV-finite character property, we simply say that  the subset $X$ of $\Spec(D)$ has {\it GV-finite character property.} Note that, in this case, $\ast$ is necessarily stable, that is $\bar{\ast}=\ast$.
 Clearly, for any domain $D$, the sets $\Max(D)$ and $\Max^{\ttts}(D)$ have GV-finite character property. Therefore, from Proposition \ref{GVFCP} and from  \cite[Corollary 2.8 and Proposition 2.9]{FiSp}, we easily deduce the following.

 \begin{corollary}\label{setGVFCP}
 Let $D$ be an  integral domain and let $\boldsymbol{\mathscr{X}}:= \{D_{P} \mid P\in X \}$ be a defining family of quotient rings of $D$ for some nonempty $X\sub\Spec(D)$.  Let $\ast_X$ be the (semi)star operation on $D$,  induced by the family of overrings
$\boldsymbol{\mathscr{X}}:= \{D_{P} \mid P\in X \}$, i.e., $\ast_X := \wedge_ {\boldsymbol{\mathscr{X}}}$.
Then the following are equivalent.
\begin{itemize}\label{}
\item[(i)] $X$  has GV-finite character property;
\item[(ii)] If $I$ is an ideal of $D$ such that $I\nsubseteq P$ for every ideal $P\in X$, then there exists $J\subseteq I$ a finitely generated ideal of $D$ such that $J\nsubseteq P$ for every ideal $P\in X$;
\item[(iii)]  $\ast_X$ is of finite type;
\item[(iv)] $X$ is quasi-compact for the Zariski topology on $\Spec(R)$.
\end{itemize}
\end{corollary}

Given  a semistar operation $\star$ on an integral domain $D$,
$D$ is called an {\it $H(\star)$-domain} \cite{FP} if  for every nonzero ideal $I$ of $D$ such that $I^\star=D$, there exists  a nonzero finitely generated ideal $J$ of $D$ such that $J\subseteq I$ and  $J^\star=D$.
Thus, given an integral domain $D$ and $X\sub\Spec(D)$ such that $\{D_{P} \mid P\in X \}$ is a defining family of quotient rings of $D$, by Proposition \ref{GVFCP} ,  $X$ has GV-finite character property if and only if  $D$ is an $H(\ast_X)$-domain, where $\ast_X$ is the (semi)star operation induced by the defining family $\{D_{P} \mid P\in X \}$ of $D$.

Note that the $H(\star)$-domains generalize in the semistar setting the $H$-domains introduced
by Glaz and Vasconcelos \cite{G-V}; more precisely, the  $H$-domains coincide with the $H(\vv)$-domains \cite[Section 2]{FP}

\medskip

The following theorem provides an algebraic version of the solution of the  problem when an essential domain is a P$\vv$MD. This problem was recently solved in  \cite{FT}
using topological methods.

\begin{theorem} \label{essentialPVMD}
Let $D$ be an integral.  Then the following are equivalent.
\begin{itemize}
\item[(i)] $D$ is a P\! $\vv$MD;
\item[(ii)] $D$ is essential and the set $ \{D_{P} \mid P\in \mathcal{E}(D) \}$ of all essential valuation overrings of $D$  has GV-finite character property.
 \item[(iii)] $D$ is essential and, for all $a,b \in D\setminus \{0\}$,    $aD \cap bD = F^{\vvs}$  for some $F \in \f(D)$   (in particular, $F \subseteq aD \cap bD$).
\end{itemize}
\end{theorem}

\begin{proof}

(i) $\Rightarrow$ (ii) Since a P$\vv$MD is an essential domain, we next show that $\mathcal{E}(D)$ has GV-finite character property. Let $I$ be an ideal of
$D$ such that $I\nsubseteq P$ for every  $P\in \Spec(D)$ such that $D_P\in \mathcal{E}(D)$ (such a prime ideal is called {\it essential prime of} $D$). Since $\Max^{\ttts}(D)\sub \mathcal{E}(D)$, $\ast_{\mathcal{E}(D)}\le \ast_{\Max^{\ttts}(D)} = \ww$.
Hence $I^{\wws}=D$. Then, there exists  a nonzero finitely generated ideal $J$ of $D$ such that $J\subseteq I$ and $J^{\wws}=D$.
But, as each essential prime ideal $P$ is such that $PD_P$ is a $\ttt$-ideal in the valuation domain $D_P$, $P$ is a $\ttt$-ideal of $D$  \cite[Lemma 3.17]{Kan} and so it is contained in a maximal $\ttt$-ideal. Thus, we get
that $J\nsubseteq P$ for every essential prime $P$. Therefore,  $\mathcal{E}(D)$  has GV-finite character property.

 (ii) $\Rightarrow$ (i) By assumption, we have $D=\bigcap \{ D_{P} \mid P\in \mathcal{E}(D)\}$. By Corollary \ref{setGVFCP},  the (semi)star operation $\ast_{\mathcal{E}(D)}$ is of finite type, so $\ast_{\mathcal{E}(D)}\le \ttt$.
 Hence, each $\ttt$-maximal
ideal is a $\ast_{\mathcal{E}(D)}$-ideal. Thus, each $\ttt$-maximal ideal is contained in an   essential prime ideal, and hence
it is an essential prime. This   proves  that $D$ is a P$\vv$MD.

 (i) $\Rightarrow$ (iii)  Recall that
$aD\cap bD=ab(a, b)^{-1}$.  Since $D$ is a P$\vv$MD,  we have $((a,b)(a,b)^{-1})^{\ttts} =D$.
By a standard argument,  we can find a finitely generated subideal $F$ of $aD\cap bD$ such that  $aD \cap bD = F^{\ttts} = F^{\vvs}$.

 (iii) $\Rightarrow$ (i) is well known   \cite[Lemma 8]{Zaf 2}.
\end{proof}

\begin{remark}
{\rm
By the above characterization, an essential domain to be a P$\vv$MD it is equivalent to the condition that the (semi)star operation induced by the defining family is of finite type, and in this case it is the $\ww$-operation.

A \texttt{P}-domain need not be a P$\vv$MD, see an example in \cite{MZ}. This shows that the defining family of localizations at
associated primes of a \texttt{P}-domain do not have in general GV-finite character property, or equivalently, the (semi)star operation induced by this defining family is not in general of finite type.
}
\end{remark}

\medskip

\noindent {\bf Acknowledgement}

\noindent The authors thanks the referee for the evaluation of the paper and for the helpful suggestions.



\begin{thebibliography}{abc}

\bibitem{And}  D.D. Anderson, Star operations induced by overrings,
Comm. Algebra {\bf 16} (1988) 2535--2553.

\bibitem{And-And} D.D. Anderson and D.F. Anderson, Some remarks on star operations and the class group, J. Pure Appl. Algebra {\bf 51} (1988), 27--33.

\bibitem{And-Cook}
D. D. Anderson and S. J. Cook, Two star-operations and their induced lattices,
Comm. Algebra {\bf 28} (2000),  2461--2475.

\bibitem{AMZ 2} D.D. Anderson, J. Mott and M. Zafrullah, Finite character
representations for integral domains, Boll. U.M.I (7) {\bf 6-B} (1992),
613--630.

\bibitem{BH}
J. Brewer and  W. Heinzer,  Associated primes of
principal ideals, Duke Math. J. {\bf 41}  (1974), 1--7.


\bibitem{de}
N. Dessagnes, Intersections d'anneaux de Mori - exemples, Portugaliae Math. {\bf 44} (1987), 379--392.


\bibitem{EFP}
S. El Baghdadi, M. Fontana, and G. Picozza,  Semistar Dedekind domains.  J. Pure Appl. Algebra {\bf 193} (2004), 27--60. 


\bibitem{FiSp}  C. A. Finocchiaro and D. Spirito, Some topological considerations on semistar operations, J. Algebra {\bf 409} (2014), 199--218.


\bibitem{FT}  C. A. Finocchiaro and F. Tartarone, On a topological characterization of Pr\"ufer $v$-multiplication domains among essential
domains, arxiv.org/pdf/1410.4037, J. Comm. Algebra (to appear).


\bibitem{FG}  M. Fontana and S. Gabelli, On the class group and the local class group of a pullback, J. Algebra
{\bf 181} (1996),  803--835.

\bibitem{FH} M. Fontana and J. Huckaba,   Localizing systems and semistar operations, in ``Non-Noetherian commutative ring theory'', 169--197, Math. Appl., {\bf 520}, Kluwer Acad. Publ., Dordrecht, 2000.

\bibitem{FJS} M. Fontana, P. Jara and E. Santos, Pr\"{u}fer $\star$-multiplication
domains and semistar operations, J. Algebra Appl. {\bf 2} (2003), 21--50.

\bibitem{Fo-Lo}
M. Fontana and K.A. Loper, Cancellation properties in ideal systems: a classification of e.a.b. semistar operations, J. Pure Appl. Algebra {\bf 213} (2009), 2095--2103.


\bibitem{FL} M. Fontana and K.A. Loper, Kronecker function rings: a general approach, in ``Ideal theo\-retic
methods in commutative algebra'' (Columbia, MO, 1999), 189Ð205, Lecture Notes in Pure
and Appl. Math., 220, Dekker, New York, 2001.

\bibitem{FL-2003} M. Fontana and K.A. Loper, Nagata rings, Kronecker function rings and related semistar operations, Comm. Algebra {\bf 31} (2003), 4775--4805.

\bibitem{FLM} M. Fontana, K.A. Loper,   and R. Matsuda,  Cancellation properties in ideal systems: an  e.a.b. not  a.b. star operation,
AJSE (Arabian Journal for Science and Engineering)--Mathematics {\bf 35} (2010), 45--49.

\bibitem{FP} M. Fontana and G. Picozza, Semistar invertibility on integral domains, Algebra Colloquium {\bf 12} (2005), 645--664.



\bibitem{Fos} R. Fossum, The divisor class group of a Krull domain,
Ergebnisse der Mathematik und ihrer grenzgebiete B. 74, Springer-Verlag,
Berlin, Heidelberg, New York, 1973.

\bibitem{G-J-S}
J.M. Garc\'ia, P. Jara, and E. Santos, Pr\"ufer $\ast$-multiplication domains and torsion theories, Comm. Algebra {\bf 27} (1999), 1275--1295.

\bibitem{Gil} R. Gilmer, Multiplicative Ideal Theory, Marcel-Dekker, New
York, 1972.

\bibitem{Gi-He}  R. Gilmer and W. Heinzer, Intersections of quotient rings of an integral domain, J. Math. Kyoto Univ. {\bf 7} (1967), 133--150

 \bibitem{G-V}
 S. Glaz and W. V. Vasconcelos, Flat ideals. II, Manuscripta Math. {\bf 22} (1977),
325--341.

\bibitem{Gri 1} M. Griffin,  Some results on $v$-multiplication rings,
Canad. J. Math. {\bf 19} (1967), 710--722.

\bibitem{HK:1998} F. Halter-Koch, Ideal Systems: An Introduction to Multiplicative Ideal Theory,
M. Dekker, New York, 1998.


\bibitem{H-O}
 W. Heinzer and J. Ohm, An essential ring which is not a $v$-multiplication ring.
Canad. J. Math. {\bf 25} (1973), 856--861.
 

\bibitem{HR}  W. Heinzer and M. Roitman, Well-centered overrings of an integral domain, J. Algebra {\bf 272} (2004), 435--455.

\bibitem{[HMM]} E. Houston, S. Malik and J. Mott, Characterizations of $\ast
$-multiplication domains, Canad. Math. Bull. {\bf 27} (1984), 48--52.

\bibitem{J:1960}  P. Jaffard, Les Syst\`emes d'Id\'eaux, Dunod, Paris, 1960.

\bibitem{Kan} B.G. Kang, Pr\"{u}fer $v$-Multiplication domains and the ring
$R[X]_{N_{v}}$, J. Algebra {\bf 123} (1989), 151--170.

\bibitem{KP} Dong J. Kwak and Young Soo Park, On $t$-flat overrings,
Chinese J. Math. {\bf 23} (1995), no. 1, 17--24.

\bibitem{MMZ} S. Malik, J. Mott and M. Zafrullah, On $t$-invertibility, Comm. Algebra {\bf 16} (1988),  149--170.

\bibitem{Ma} H. Matsumura, Commutative ring theory, Cambridge University Press,  Cambridge, 1986.




\bibitem{MZ} J. Mott and M. Zafrullah, On Pr\"ufer $v$-multiplication domains, Manuscripta Math.  {\bf 35}
(1981), 1--26.



\bibitem{Pi} G. Picozza, Semistar operations and multiplicative ideal theory, Ph.D. Thesis, Universit\`a degli Studi ``Roma Tre'', Rome, 2004.


\bibitem{Q}
J. Querr\'e, Intersections d'anneaux int\`egres, J. Algebra {\bf 43} (1976), 55--60.

\bibitem{Ri}
F. Richman, Generalized quotient rings, Proc. Amer. Math. Soc. {\bf 16} (1965), 794--799.



\bibitem{U}
H. Uda, LCM-stableness in ring extensions,  Hiroshima Math. J. {\bf 13} (1983), 357--377.

\bibitem{WM} F. Wang and R. McCasland, On $w$-modules over strong Mori
domains, Comm. Algebra {\bf 25} (1997) 1285--1306.

\bibitem{Zaf 2} M. Zafrullah, On finite conductor domains,
Manuscripta Math. {\bf 24} (1978), 191--203.

\bibitem{Zaf 7} M. Zafrullah, Ascending chain conditions and star operations,
Comm. Algebra {\bf 17}(1989), 1523--1533.

\bibitem{Zaf 10} M. Zafrullah, Putting $t$-invertibility to use, in ``Non-Noetherian commutative ring theory'', 429-457, Math. Appl., 520, Kluwer
Acad. Publ., Dordrecht, 2000.

\end{thebibliography}
\end{document}